\documentclass[12pt]{amsart}
 \usepackage{amsfonts,amssymb}
  \usepackage{enumitem}
 \setlist[itemize]{noitemsep,nolistsep}
 \usepackage[toc,page]{appendix}
\usepackage{mathrsfs}
\usepackage{enumitem}
\usepackage{array}
\usepackage[all,ps,cmtip,rotate]{xy} 
 \usepackage{color}
\usepackage{hyperref}
\usepackage{bbold} 
 
\def\Z{{\bf Z}}

\def\C{{\bf C}}

\def\P{{\bf P}}
\def\PP{{\bf P}}

\def\cI{\mathscr{I}}

\def\cO{\mathscr{O}}

\def\cP{\mathscr{P}}
\def\cR{{\mathscr{R}}}

\def\cC{\mathscr{C}}
\def\cM{\mathscr{M}}

\def\cK{\mathscr{K}}
\def\cT{{\mathscr{T}}}

\def\cU{\mathscr{U}}

\def\BBQ{{\overline{Q}}}

\def\llra{\hbox to 10mm{\rightarrowfill}}
\def\lllra{\hbox to 15mm{\rightarrowfill}}

\def\llla{\hbox to 10mm{\leftarrowfill}}
\def\lllla{\hbox to 15mm{\leftarrowfill}}

\def\dra{\dashrightarrow}

\def\hra{\hookrightarrow}

\def\isom{\simeq}

\def\ie{\hbox{i.e.}}

 \def\vide{\varnothing}
  \def\emptyset{\varnothing}

\DeclareMathOperator{\isomto}{\stackrel{{}_{\scriptstyle\sim}}{\to}}

\DeclareMathOperator{\codim}{codim}

\DeclareMathOperator{\ev}{ev}

\DeclareMathOperator{\Gr}{\mathsf{Gr}}

\DeclareMathOperator{\CGr}{\mathsf{CGr}}

\DeclareMathOperator{\Fl}{\mathsf{Fl}}
\DeclareMathOperator{\Hilb}{Hilb}

\DeclareMathOperator{\Ker}{Ker}

\DeclareMathOperator{\Sym}{\mathsf S}

\def\llra{\hbox to 10mm{\rightarrowfill}}
\def\lllra{\hbox to 15mm{\rightarrowfill}}

\def\bw#1#2{\textstyle{\bigwedge\hskip-0.9mm^{#1}}\hskip0.2mm{#2}}
\def\sbw#1#2{{\bigwedge\hskip-0.9mm^{#1}}\hskip0.1mm{#2}}

\newtheorem{lemm}{Lemma}[section]
\newtheorem{theo}[lemm]{Theorem}
\newtheorem{coro}[lemm]{Corollary}
\newtheorem{prop}[lemm]{Proposition}

\theoremstyle{remark}

\def\unit{\mathbb{1}}
\def\LL{\mathbb{L}}
\def\MM{\mathbb{M}}

\def\BW{{\overline{W}}}

\newcommand{\cone}[1]{\mathsf{C}_{#1}}

\DeclareMathOperator{\OGr}{\mathsf{OGr}}
\DeclareMathOperator{\OFl}{\mathsf{OFl}}
\DeclareMathOperator{\Bl}{\mathrm{Bl}}

\newcommand{\rS}{{\mathscr{S}}}

\newcommand{\reg}{\mathrm{reg}}


\begin{document}
\title 
{On the cohomology of Gushel--Mukai sixfolds}

\author[O.~Debarre]{Olivier Debarre}
\address{\parbox{0.9\textwidth}{Univ  Paris Diderot, \'Ecole normale su\-p\'e\-rieu\-re, PSL Research University,
\\[1pt] 
CNRS, D\'epar\-te\-ment Math\'ematiques et Applications
\\[1pt]
45 rue d'Ulm, 75230 Paris cedex 05, France}}
\email{{olivier.debarre@ens.fr}}

 \author[A. Kuznetsov]{Alexander Kuznetsov}
\address{\parbox{0.9\textwidth}{Algebraic Geometry Section, Steklov Mathematical Institute,
\\[1pt]
8 Gubkin str., Moscow 119991, Russia
\\[5pt]
The Poncelet Laboratory, Independent University of Moscow
\\[5pt]
Laboratory of Algebraic Geometry, National Research University Higher School of Economics, Russian Federation}}
\email{{\tt  akuznet@mi.ras.ru}}
 
\thanks{A.K. was partially supported by the Russian Academic Excellence Project ``5-100'', by RFBR grants 14-01-00416, 15-01-02164, 15-51-50045, and by the Simons foundation.}

\def\setminus{\smallsetminus}
\def\cong{\isom}

\newcommand{\blue}[1]{\textcolor{blue}{#1}}
\newcommand{\red}[1]{{\color{red}{#1}}}

\begin{abstract}
We provide a   stable rationality construction for some smooth complex Gushel--Mukai varieties of dimension 6.
As a consequence, we  {compute} the integral singular cohomology of any smooth Gushel--Mukai sixfold  {and in particular, show that it} is torsion-free.
\end{abstract}

\maketitle

\section{Introduction}

A smooth (complex) Gushel--Mukai (GM for short) variety of dimension $n$ is a smooth dimensionally transverse intersection
\begin{equation*}
X = \CGr(2,V_5) \cap \P(W) \cap Q \subset \P(\C \oplus \bw2V_5),
\end{equation*}
of a cone with vertex $ {\mathsf{v}}$ over the Grassmannian $\Gr(2,V_5)$ of two-dimensional subspaces in a five-dimensional vector space $V_5$
with a linear subspace $\P(W) \cong \P^{n+4}$ and a quadratic hypersurface $Q \subset \P(W)$ (\cite[{Definition~2.1}]{DK1}).

There are two types of smooth GM varieties.
If the linear space $\P(W)$ does not contain the vertex~${\mathsf{v}}$, one can rewrite $X$  as $ \Gr(2,V_5) \cap \P(W) \cap Q$,
\ie, $X$ is a quadratic section of a linear section of the Grassmannian.
If, on the  contrary, $\P(W)$ contains the vertex,   $\CGr(2,V_5) \cap \P(W)$ is a cone over a linear section of the Grassmannian,
and since the quadric $Q$ does not pass through ${\mathsf{v}}$ (since $X$ is smooth), the projection from ${\mathsf{v}}$ identifies $X$
with a double cover of this linear section branched along a GM variety of dimension $n-1$.

GM varieties of the first type are called \emph{ordinary}  and those of the second type are called \emph{special}.\
When $n \le 5$, the two types of GM varieties belong to the same deformation family; when $n = 6$, all GM varieties are special.

Since a smooth ordinary GM variety    is a complete intersection of ample divisors in  {the Grassmannian}~$\Gr(2,V_5)$, the Lefschetz Hyperplane Theorem describes all its {singular} cohomology groups except for the middle one 
and implies that these   groups are all torsion-free.
By a deformation argument, the same results hold  for all smooth GM varieties of dimension at most~5.
This argument does not work in dimension~6 since there are no ordinary GM sixfolds.

Using analogous arguments, one can determine the cohomology groups $H^k(X,\Z)$ of a smooth (special) GM sixfold $X$ except for $k=6$ or 7  {(see~\cite[Proposition~3.3]{DK2})}.\
Whether the (isomorphic) torsion subgroups of  $H^6(X,\Z)$ and $H^7(X,\Z)$ are zero seems to be a   question not easily answered by standard tools.

The goal of this paper is to show that $H^\bullet(X,\Z)$ is torsion-free for any smooth GM sixfold~$X$ by using a geometrical approach.\
It is natural to attack this question by using the fact that~$X$ is rational (\cite[Proposition~4.2]{DK1}): a good factorization of a birational isomorphism~$X\dra \P^6$  would lead to a description of $H^6(X,\Z)$.\
However, the birational isomorphism described in~\cite[Proposition~4.2]{DK1} (and   any other construction we are aware of) is quite complicated and does not lead to a simple factorization.

Our idea is  {to} use instead a stable rationality construction.
We show  that a natural \mbox{$(\P^3\times \P^1)$}-fibration over a specific  GM sixfold $X$ 
is dominated (birationally) by an explicit iterated blow up of the product  of $\P^4$ with a smooth six-dimensional quadric.
This allows us to embed the group $H^\bullet(X,\Z)$ into a direct sum of Tate twists of $\Z$, the cohomology of a K3 surface~$Y$, and that of an ample divisor in the Hilbert square of $Y$.\
This implies that $H^\bullet(X,\Z)$ is torsion-free (see Theorem~\ref{theorem:main} for   details).
By a deformation argument, the cohomology of all smooth GM sixfolds is torsion-free.

The K3 surface $Y$  appearing above  is itself a GM variety and is related to the sixfold~$X$ by a generalization  (\cite[{Definition~4.7}]{KP}) of the duality discussed in~\cite[Section~{3.}6]{DK1}.
In fact, we work in the opposite direction: we start from a sufficiently general GM surface $Y$  and   take~$X$ to be its generalized dual GM sixfold.\
Then we construct $\P^3$-fibrations over $X$ and~$Y$ that come with maps to a six-dimensional quadric with fibers   mutually orthogonal linear sections 
of the dual Grassmannians $\Gr(2,V_5)$ and $\Gr(2,V_5^\vee)$ of respective codimensions 3 and 7.\
A smooth linear section of $\Gr(2,V_5)$ of codimension 3 is a quintic del Pezzo threefold.\
It is a classical observation (already known to Castelnuovo) that the projectivization of the tautological bundle on it is   the blow up of $\P(V_5)$ along a  projection of a Veronese surface.\
Generalizing this to mildly singular quintic del Pezzo threefolds and performing the construction in a family  gives the required stable rationality construction for~$X$.

The paper is organized as follows.\
In Section~\ref{section:gm-duality}, we describe the construction of a smooth GM sixfold $X$ from a GM surface $Y$ and a relation between them.\
In particular, we construct a quintic del Pezzo fibration structure on a $\P^3$-bundle over $X$ whose degeneration is controlled by a $\P^3$-fibration over $Y$.\
In Section~\ref{section:families}, we describe a rationality construction for (a $\P^1$-bundle over) a family of quintic del Pezzo threefolds.\
In Section~\ref{section:gm-6folds}, we describe in detail the fibration constructed in Section~\ref{section:gm-duality} and deduce torsion-freeness of $H^\bullet(X,\Z)$.\
We also speculate about the structure of the Chow motive of $X$ and its relation to that of $Y$,
and as a byproduct give a simple computation of the Hodge numbers of $X$.

\section{GM surfaces and their generalized dual GM sixfolds}
\label{section:gm-duality}

\subsection{A general GM surface}

Let $V_5$ be a five-dimensional vector space and set $V'_5 := V_5^\vee$.
Let $W' \subset \bw2V'_5$ be a   subspace of codimension 3 such that
\begin{equation*}
M' := \Gr(2,V'_5) \cap \P(W')
\end{equation*}
is a smooth (quintic del Pezzo) threefold and let $Q' \subset \P(W')$ be a  smooth  quadric such that
\begin{equation*}
Y := M' \cap Q'
\end{equation*}
is a smooth GM (K3) surface (\cite[Section~2.3]{DK1}).

Consider the Hilbert squares  
 and natural maps
\begin{equation*}
\Hilb^2(Y) \hra \Hilb^2(M') \hra \Hilb^2(\P(W')) \to \Gr(2,W')
\end{equation*}
(the first two maps are induced by the embeddings $Y \hra M' \hra \P(W')$  and the last map takes a length-2 subscheme {$\xi$} to the unique line $\langle\xi\rangle\subset \P(W')$ that contains it.\
Consider also the isotropic Grassmannian $\OGr_{Q'}(2,W') \subset \Gr(2,W')$ parameterizing lines  {in $\P(W')$} contained in~$Q'$  and define
\begin{equation}\label{mma}
 \Hilb^2_{Q'}(Y) := \Hilb^2(Y) \times_{\Gr(2,W')} \OGr_{Q'}(2,W') . 
\end{equation}
This threefold  parameterizes   length-2 subschemes $\xi$ of $Y$ whose linear span $\langle\xi\rangle$  is contained in $Q'$ and it will play an important role in  the construction of Section~\ref{section:gm-6folds}.

\begin{lemm}\label{lemma:s2q}
When $Q'$  is general, $Y = M' \cap Q'$ is a smooth surface containing neither lines  nor conics and the subscheme $\Hilb^2_{Q'}(Y) \subset \Hilb^2(Y)$ is a smooth ample divisor. 
\end{lemm}

\begin{proof}
Since $Q'$ is general, the surface $Y$ is a general polarized K3 surface of degree 10 in $\P^6$, hence contains neither lines  nor conics.

Define   $\Hilb^2_{Q'}(M') := \Hilb^2(M')   \times_{\Gr(2,W')}  \OGr_{Q'}(2,W')$.\
  {A bit surprisingly,}  
we have
\begin{equation*}
\Hilb^2_{Q'}(Y) = \Hilb^2_{Q'}(M'). 
 \end{equation*}
Indeed, if $\xi$ is a subscheme of $M'$ such that $\langle \xi \rangle \subset Q'$, then $\xi \subset M' \cap Q' = Y$.
This gives one embedding  and the other   is induced by the embedding $Y \hra M'$.

Let $\cP$ be the tautological {rank-2} bundle on $\Gr(2,W')$.\
The vector bundle $\Sym^2\!\cP^\vee$ is generated by the space~$\Sym^2\!W^{\prime\vee}$ of global sections and the zero-locus of a section $Q' \in \Sym^2\!W^{\prime\vee}$ is $\OGr_{Q'}(2,W') \subset \Gr(2,W')$.\
By base change, the same is true on $ \Hilb^2(M')$, hence for $Q'$ general,  the zero-locus $\Hilb^2_{Q'}(M') \subset \Hilb^2(M')$ is smooth of codimension~3.\
Thus for general~$Q'$, the subscheme $\Hilb^2_{Q'}(Y) \subset \Hilb^2(Y)$ is a smooth divisor.

To prove ampleness, consider the   double cover 
\begin{equation*}
\sigma \colon \Bl_{\Delta(Y)}(Y \times Y) \xrightarrow{\ 2:1\ } \Hilb^2(Y)
\end{equation*}
from the blow up of $Y\times Y$ along the diagonal.
Since $\sigma$ is finite, it is enough to show that the divisor  $\sigma^*(\Hilb^2_{Q'}(Y)) $ is ample.
The dual of the pullback ${\widetilde\cP}$ of the bundle $\cP$ to the blow up $\Bl_{\Delta(Y)}(Y \times Y)$ fits into an exact sequence
\begin{equation*}
0 \to {\widetilde\cP}^\vee \to \cO(h_1) \oplus \cO(h_2) \to \epsilon_*\cO(h) \to 0,
\end{equation*}
where $h_1$ and $h_2$ are the pullbacks of the hyperplane classes of the factors of $Y \times Y$, $\epsilon$ is the embedding of the exceptional divisor $\P(T_Y) \subset \Bl_{\Delta(Y)}(Y \times Y)$,
and $h$ is the pullback of the hyperplane class of $Y$ via the projection $\P(T_Y) \to Y$.\
In particular, 
\begin{equation*}
c_1({\widetilde\cP}^\vee) = h_1 + h_2 - e,
\end{equation*}
where $e$ is the class of the exceptional divisor.\
Since $Y$ is an intersection of quadrics and contains no lines, the map $\Hilb^2(Y) \to \Gr(2,W')$ 
 is an embedding.\ 
The ample class $c_1({ \cP}^\vee)$   therefore remains ample on~$\Hilb^2(Y)$, hence also the class $h_1 + h_2 - e$ on $\Bl_{\Delta(Y)}(Y \times Y)$.

Furthermore,   the symmetric square of ${\widetilde\cP}^\vee$ fits into exact sequence
\begin{equation*}
0 \to \cO(h_1 + h_2 - e) \to \Sym^2\!{\widetilde\cP}^\vee \to \cO(2h_1) \oplus \cO(2h_2) \to \epsilon_*\cO(2h) \to 0.
\end{equation*}
The image of the global section $Q'$ of $\Sym^2\!{\widetilde\cP}^\vee$ in $\cO(2h_1) \oplus \cO(2h_2)$ vanishes (since $Y \subset Q'$), hence the class of $\sigma^*(\Hilb^2_{Q'}(Y))$ in $\Bl_{\Delta(Y)}(Y \times Y)$ equals $h_1 + h_2 - e$.\
As we saw, this class is ample. 
 \end{proof}

 {In the rest of the paper, we make the following generality assumptions on the subspace $W' \subset \bw2V'_5$ and the quadric $Q' \subset \P(W')$:}
\begin{equation}\label{eq:assumptions}
\begin{aligned}
& \bullet \text{the intersection $M' = \Gr(2,V'_5) \cap \P(W')$ is a smooth threefold;}\\
& \bullet \text{the quadratic hypersurface $Q' \subset \P(W')$ is smooth;}\\
& \bullet \text{the intersection $Y = M' \cap Q'$ is a smooth surface containing neither lines, nor conics;}\\
& \bullet \text{the divisor $\Hilb^2_{Q'}(Y) \subset \Hilb^2(Y)$ is smooth  and   ample.}
\end{aligned}
\end{equation}

\subsection{Dual GM sixfold}\label{subsection:dual-gm}

As in the previous section,     $W' \subset \bw2V'_5 = \bw2V^\vee_5$ is a subspace of codimension 3 and   $Q' \subset \P(W')$    a smooth quadratic hypersurface.\
Considering it as a quadric (of codimension~4) in~$\P(\bw2V^\vee_5)$, we denote by 
\begin{equation*}
Q_0 := Q^{\prime\vee} \subset \P(\bw2V_5)
\end{equation*}
its projective dual.
It can be {described} as follows.

Let $K := {W}^{\prime\perp} \subset \bw2V_5$ be the orthogonal of~$W'$, so that $\BW_0 := (\bw2V_5)/K$  {is isomorphic to} $W^{\prime\vee}$.
Since $Q'$ is smooth, the corresponding quadratic form defines an isomorphism $W' \isomto W^{\prime\vee} = \BW_0$.\
Its inverse is an isomorphism $\BW_0 \isomto W' = \BW_0^\vee$ and thus defines a smooth quadratic hypersurface $\BBQ_0 \subset \P(\BW_0)$.\
Then $Q_0$ is the cone $ \cone{\P(K)}\BBQ_0$ over $\BBQ_0 \subset \P(\bw2V_5/K)$ with vertex $\P(K)$.\
In particular, it is a {quadratic hypersurface} of corank 3 in $\P(\bw2V_5)$.

We define a GM sixfold $X$ as the double cover 
\begin{equation*}
X \xrightarrow{\ 2:1\ } \Gr(2,V_5)
\end{equation*}
branched along the ordinary GM fivefold  $X_0 := \Gr(2,V_5) \cap Q_0$.

\begin{lemm}\label{lemma:6fold}
Under the assumptions~\eqref{eq:assumptions},
 $X$ is a smooth GM sixfold.
\end{lemm}

\begin{proof}
 The argument of~\cite[Proposition~3.26]{DK1} shows that $X$ is  generalized dual to $Y$ (see also~\cite[{Definition~4.7}]{KP}).\
In particular, if $A \subset \bw3V_6$ is the Lagrangian subspace corresponding to $X$,  the Lagrangian subspace corresponding to $Y$ is $A^\perp \subset \bw3V_6^\vee$.\
Since both $Y$ and  $M'$ are smooth, $A^\perp$ contains no decomposable vectors by~\cite[Theorem~3.14]{DK1}.\
It follows   that $A$ has no decomposable vectors either, hence $X$ is smooth, again by~\cite[Theorem~3.14]{DK1}.
\end{proof}

 The GM sixfolds $X$ constructed in this way  form a family of codimension 1 in the coarse moduli space of all GM sixfolds.

\subsection{The quadrics}

According to~\cite[Lemma~2.31]{DK1}, the restriction from $\C \oplus \bw2V_5$  to its subspace $\bw2V_5  $ defines an isomorphism between the space  of quadratic equations of $X $ in $ \P(\C \oplus \bw2V_5)$ and the space  of quadratic equations of $X_0 $ in $ \P(\bw2V_5)$.\
Let $Q \subset \P(\C \oplus \bw2V_5)$ be the quadric that corresponds to $Q_0$.\ 
The kernel space of $Q$ equals the kernel space $K$ of $Q_0$ so, setting $\BW := (\C \oplus \bw2V_5)/K = \C \oplus \BW_0$, we have
\begin{equation*}
Q = \cone{\P(K)}\BBQ,
\end{equation*}
where $\BBQ \subset \P(\BW)$ is a non-degenerate quadric.\
So we have   diagrams of spaces and quadrics 
\begin{equation}\label{diagram-quadrics}
\vcenter{\xymatrix{
Q_0 \ar@{^{(}->}[r] \ar@{-->}[d] & 
Q \ar@{-->}[d] \\
\BBQ_0 \ar@{^{(}->}[r] & 
\BBQ
}}
\qquad \subset \qquad
\vcenter{\xymatrix{
\P(\bw2V_5)\, \ar@{^{(}->}[r] \ar@{-->}[d] & 
\P(\C \oplus \bw2V_5) \ar@{-->}[d] \\
\P(\BW_0)\, \ar@{^{(}->}[r] & 
\P(\BW),
}}
 \end{equation}
where the horizontal arrows are smooth hyperplane sections  and the vertical arrows are linear projections from $\P(K)$.\
The quadrics $Q' \subset \P(W')$ and $\BBQ_0 \subset \P(\BW_0)$ are projectively dual.\
 {Thus} $Q'$, $\BBQ_0$, and $\BBQ$ are smooth quadric hypersurfaces in $\P^6$, $\P^6$, and $\PP^7$ respectively, while $Q_0$ and $Q$ are quadric hypersurfaces of corank 3 in $\P^9$ and $\P^{10}$ respectively.

In particular, $\BBQ\subset \P(\BW)$ is a smooth quadric of dimension~6.\
We choose one of the two connected components, $\OGr_+(4,\BW)$, of the corresponding orthogonal Grassmannian  and we consider the corresponding connected component $\OFl_+(1,4;\BW)$ of the orthogonal flag variety 
\begin{equation*}
\xymatrix{
& \OFl_+(1,4;\BW) \ar[dl]_{p_\BBQ} \ar[dr]^{q_\BBQ} \\
\BBQ && \OGr_+(4,\BW),
}
\end{equation*}
where $p_\BBQ$ and $q_\BBQ$ are both $\P^3$-fibrations: $q_\BBQ$ is the projectivization of the tautological  {rank-4} bundle $\cT_4$ on $\OGr_+(4,\BW)$ and $p_\BBQ$ is the projectivization of the spinor bundle $\rS_\BBQ$.\  
 The scheme 
\begin{equation}\label{eq:define-b}
B := \OGr_+(4,\BW)
\end{equation}
is isomorphic to a 6-dimensional quadric (it will be the base for the family of quintic del Pezzo threefolds considered later).\
The above diagram takes the form 
\begin{equation}\label{diagram:bq}
\vcenter{\xymatrix{
& \P(\rS_\BBQ) \ar[dl]_{p_\BBQ} 
\ar@{=}[r] & \P_B(\cT_4) 
\ar[dr]^{q_\BBQ} \\
\BBQ &&& B.
}}
\end{equation}
We transfer this diagram to the other quadrics  by using the maps in~\eqref{diagram-quadrics}
\begin{equation*}
\vcenter{\xymatrix@C=1.1em{
& \P(\rS_{\BBQ_0}) \ar[dl]_-{p_{\BBQ_0}} 
\ar@{=}[r] & \P_B(\cT_3) 
\ar[dr]^-{q_{\BBQ_0}} \\
\BBQ_0 &&& B
}}
\qquad
\vcenter{\xymatrix@C=1.1em{
& \P(\rS_{Q_\reg})\, \ar[dl]_-{p_{Q}} 
\ar@{^{(}->}[r] & \P_B((K\otimes \cO_B) \oplus \cT_4)
\ar[dr]^-{q_{Q}} \\
Q_{\reg} &&& B,
}}
\end{equation*}
where $Q_\reg = Q \setminus \P(K)$ is the smooth locus of $Q$, 
$\rS_{\BBQ_0}$ and $\rS_{Q_\reg}$ are the pullbacks of $\rS_\BBQ$ via the maps $\BBQ_0 \hookrightarrow \BBQ$ and $Q_\reg \twoheadrightarrow \BBQ$, and 
\begin{equation*}
\cT_3 = \cT_4 \cap (\BW_0 \otimes \cO_B) 
\end{equation*}
is a rank-3 bundle on $B$.
This identifies $B$ with the orthogonal Grassmannian $\OGr(3,\BW_0)$ of the  {5-dimensional} quadric $\BBQ_0$, and $\cT_3$ with its tautological bundle.

The maps $p$ in the diagrams are all $\P^3$-fibrations.
The map $q_{\BBQ_0}$ is a $\P^2$-fibration  {and} the map $q_Q$ is a $\P^6$-fibration.

The projective duality between $\BBQ_0$ and $Q'$ produces a spinor bundle $\rS_{Q'}$ on $Q'$ and a diagram
\begin{equation}\label{diagram:qp}
\vcenter{\xymatrix{
& \P(\rS_{Q'}) \ar[dl]_-{p_{Q'}} 
\ar@{=}[r] & \P_B(\cR) 
\ar[dr]^-{q_{Q'}} \\
{Q'} &&& B,
}}
\end{equation}
where 
\begin{equation}\label{eq:define-j3}
\cR := ((K\otimes \cO_B) \oplus \cT_4)^\perp \subset W' \otimes \cO_B
\end{equation} 
is a rank-3 subbundle and we use the natural identification $K^\perp = W'$.

We describe  in the next lemma the restrictions of diagrams~\eqref{diagram:bq} and~\eqref{diagram:qp} to the {GM varieties} $X \subset Q_\reg$ and $Y \subset Q'$.\ 
 We set $\rS_X :=\rS_{Q_\reg}\vert_X$ and $\rS_Y := \rS_{Q'}\vert_Y$, and we denote by $p_X$, $q_X$, $p_Y$, and $q_Y$ the restrictions of the maps $p_Q$, $q_Q$ to~$\P(\rS_X)$, and $p_{Q'}$, $q_{Q'}$ to $\P(\rS_Y)$.

\begin{lemm}\label{lemma:fibers-q}
Assume that \eqref{eq:assumptions} holds  {and consider the diagram}
 \begin{equation}
\vcenter{\xymatrix{
& \P(\rS_X) \ar[dl]_{p_X} \ar[dr]^{q_X} 
&& \P(\rS_Y) \ar[dl]_{q_Y} \ar[dr]^{p_Y}
\\
X && B && Y.
}}
\end{equation} 
The maps $p_X$ and $p_Y$ are $\P^3$-fibrations.\
Moreover, if ${R}_b \subset W' \subset \bw2V_5^\vee$ is the fiber of  $\cR$ at a point $b \in B$ and ${R}_b^\bot \subset \bw2V_5$ its orthogonal, we have
\begin{equation*}
q_X^{-1}(b) \cong \Gr(2,V_5) \cap \P({R}_b^\perp)
\qquad \text{and}\qquad
q_Y^{-1}(b) \cong \Gr(2,V^\vee_5) \cap \P({R}_b).
\end{equation*}
 The map $q_Y$ is finite 
  and the map $q_X$ is flat of relative dimension~$3$.
\end{lemm}
 
\begin{proof}
The description of the fibers is straightforward: since $Y = \Gr(2,V_5^\vee) \cap Q'$, the fibers of~$q_Y$ are the intersections of $\Gr(2,V_5^\vee)$ with the fibers $\P(R_b)$ of $q_{Q'}$; the case of $q_X$ is analogous. 

Any fiber $q_Y^{-1}(b)$ is finite: if not, since it is an intersection of quadrics in  $\P({R}_b) \cong \P^2$, it contains a line or a conic;\
but it is contained in the GM surface $Y$, and this contradicts~\eqref{eq:assumptions}.\
The proper morphism $q_Y$  is  therefore finite.\ Finally, the map $q_X$ is flat by Lemma~\ref{lemma:flatness} below. 
\end{proof}

By Lemma~\ref{lemma:fibers-q}, the fibers of the map $\P(\rS_X) \to B$ are linear sections of $\Gr(2,V_5)$ of codimension 3.
A general fiber is a smooth quintic del Pezzo threefold  and   the discriminant locus of   $q_X$ is  $q_Y(\P(\rS_Y)) \subset B$  ({see} \cite[Proposition~2.22]{DK1}).

\section{Families of quintic del Pezzo threefolds}
\label{section:families}

In this section, we discuss a stable rationality construction for general families of quintic del Pezzo threefolds which we will later apply to the family constructed in Section~\ref{section:gm-duality}. 

\subsection{Families of quintic del Pezzo threefolds}\label{se31}

 Let $B$ be a Cohen--Macaulay scheme and let $\cR \subset \bw2V_5^\vee  \otimes \cO_B$ be a rank-3 subbundle (in Section~\ref{section:gm-6folds}, we will take $B$ to be 
the 6-dimensional quadric~\eqref{eq:define-b} and define the subbundle $\cR$ by~\eqref{eq:define-j3}) with orthogonal complement 
\begin{equation*}
\cR^\perp := \Ker(\bw2V_5 \otimes \cO_B \twoheadrightarrow \cR^\vee) \subset \bw2V_5 \otimes \cO_B.
\end{equation*}
The next lemma shows that the map 
\begin{equation*}
\cM := \P_B(\cR^\perp) \times_{\P({\sbw2V_5)}} \Gr(2,V_5) \to B 
\end{equation*}
is a family of quintic del Pezzo threefolds.
For $b \in B$, we denote by~${R}_b \subset \bw2V_5^\vee$ and ${R}_b^\perp \subset \bw2V_5$ the fibers of $\cR$ and $\cR^\perp$ at $b$.

\begin{lemm}\label{lemma:flatness}
Assume that the intersection $\P({R}_b) \cap \Gr(2,V_5^\vee)$ is finite for every   $b \in B$ and empty for general $b$.
The map $\cM \to B$ is then flat of relative dimension~$3$ with general fiber a smooth quintic del Pezzo threefold.
\end{lemm}

\begin{proof}
The fiber $\cM_b $ of $\cM$ over a point $b \in B$ is $\P({R}_b^\perp) \cap  \Gr(2,V_5)$.\
To show that  {this intersection is}
 dimensionally transverse, we  choose a line $\P^1 = \P({R}') \subset \P({R}_b) = \P^2$ which does not intersect $\Gr(2,V_5^\vee)$ (this is possible since $\P({R}_b)\cap \Gr(2,V_5^\vee) $ is finite).\
 By~\cite[Proposition~2.22]{DK1}, the  intersection $\Gr(2,V_5) \cap \P({R}^{\prime\perp})$ is then smooth and dimensionally transverse   and the  intersection $\Gr(2,V_5) \cap \P({R}_b^\perp)$ is a hyperplane section, hence is also dimensionally transverse.\
The flatness of the map $\cM \to B$   easily follows.

Finally, when $\P({R}_b) \cap \Gr(2,V_5^\vee)$ is empty, $\cM_b $ is a smooth quintic del Pezzo threefold (again by ~\cite[Proposition~2.22]{DK1}).
\end{proof}

Consider the following composition  
\begin{equation}\label{eq:varphi}
\varphi \colon \cR \hookrightarrow \bw2V_5^\vee \otimes \cO_{{B \times \P(V_5)}} \twoheadrightarrow \Omega_{{\P(V_5)}}(2)
\end{equation}
of morphisms on $B \times \P(V_5)$ (we  omit  the pullback notation for the projections of $B \times \P(V_5)$ to the factors).
 {As we will see later, the geometry of $\cM$ can be described in terms of $\varphi$ and its degeneracy loci.\ 
In what follows, we write $D_k(\varphi) \subset B \times \P(V_5)$ for the corank $\ge k$ locus.}

To analyze  {$D_k(\varphi)$,}
the following observation  {is} useful.\ 
On $\P(V_5)$, there is a canonical exact sequence
\begin{equation*}
0 \to \Omega_{\P(V_5)}^2(2) \to \bw2V_5^\vee \otimes \cO_{\P(V_5)} \to \Omega_{\P(V_5)}(2) \to 0 
\end{equation*}
whose restriction to a point $v \in \P(V_5)$ is the exact sequence
\begin{equation*}
0 \to \bw2(v^\perp) \to \bw2V_5^\vee \to v^\perp \to 0,
\end{equation*}
where the second map is the contraction with $v$.
 {It follows that
\begin{equation}\label{eq:ker-varphi-bv}
\Ker (\varphi_{b,v}) = {R}_b \cap \bw2(v^\perp),
\end{equation}
where $\varphi_{b,v}$ is the fiber of $\varphi$ at $(b,v) \in B \times \P(V_5)$.}

\begin{lemm}\label{lemma:d3}
Assume  that for every point $b \in B$, the intersection $\P(R_b) \cap \Gr(2,V_5^\vee)$ 
 is finite.
Then $D_3(\varphi) = \emptyset$.
\end{lemm}

\begin{proof}
Assume $(b,v) \in D_3(\varphi)$.\
By~\eqref{eq:ker-varphi-bv}, we have inclusions ${R}_b \subset \bw2(v^\perp) \subset \bw2V_5^\vee$,
 {hence} $\P({R}_b) \cap \Gr(2,V_5^\vee) = \P({R}_b) \cap \Gr(2,v^\perp)$.\
Since 
 $\Gr(2,v^\perp) \subset \P(\bw2(v^\perp))$ is a quadric hypersurface, this intersection is a conic in $\P({R}_b)$.\
This contradicts the assumption  {of finiteness}; therefore, $D_3(\varphi) = \emptyset$.
\end{proof}

\subsection{A stable rationality construction}
We keep the  assumptions and notation of Section~\ref{se31}.\
In addition, we denote by $\cU \subset V_5 \otimes \cO_{\!\cM}$ the pullback of the tautological rank-2 bundle on~$\Gr(2,V_5)$ to~$\cM$ and consider the natural map 
\begin{equation}\label{eq:f}
f \colon \P_{\!\cM}(\cU) \to B \times \P(V_5).
\end{equation}
The next proposition shows that under appropriate assumptions on $\varphi$, this map  is birational.

\begin{prop}\label{proposition:quintic-fibration}
Assume $\codim(D_k(\varphi)) \ge k+1$ for all $k \ge 1$.\ 
The morphism $f$ is the blow up of $D_1(\varphi)\subset B \times \P(V_5)$ and is 
 a $\P^k$-fibration over $D_k(\varphi) \setminus D_{k+1}(\varphi)$. 
\end{prop}

\begin{proof}
Recall that $\P_{\Gr(2,V_5)}(\cU) \cong \Fl(1,2;V_5) \cong \P(T_{\P(V_5)}(-2))$ and the pullback of the hyperplane class of $\Gr(2,V_5)$ is 
the relative hyperplane class for $\P(T_{\P(V_5)}(-2))$ over $\P(V_5)$.\ 
Therefore,
\begin{equation*}
\P_\cM(\cU) \cong 
\P_{\Gr(2,V_5)}(\cU) \times_{\P(\sbw2V_5)} \P_B(\cR^\perp) \cong
\P(T_{\P(V_5)}(-2)) \times_{\P(\sbw2V_5)} \P_B(\cR^\perp)
\end{equation*}
is the zero-locus of the section of the vector bundle  $\cR^\vee \otimes \cO(1)$ on $\P_{B \times \P(V_5)}(T_{\P(V_5)}(-2))$ induced by $\varphi^\vee$.\
The  {first} part then follows from \cite[Lemma~2.1]{K16} applied to $\varphi^\vee$.

The second part follows from the fact that the preimage of $D_k(\varphi) \setminus D_{k+1}(\varphi)$ is the projectivization of $\Ker(\varphi^\vee\vert_{D_k(\varphi) \setminus D_{k+1}(\varphi)})$,
which is locally free of rank $k+1$. 
\end{proof}

The blow up description of Proposition~\ref{proposition:quintic-fibration} has one drawback: the center   $D_1(\varphi)$ of the blow up  is usually singular.\ 
We include it in a diagram of blow ups of other simpler   loci.

\begin{lemm}\label{lemma:blow ups}
 There is a commutative diagram
 \begin{equation}\label{eq:blowups}
\vcenter{\xymatrix@C=0pt{
\Bl_{\pi_1^{-1}(D_2(\varphi))}(\Bl_{D_1(\varphi)}(B \times \P(V_5)))
  \ar[d]^-{\pi'_2}   \ar@{=}[rr]&& 
\Bl_{\pi_2^{-1}(D_1(\varphi))}(\Bl_{D_2(\varphi)}(B \times \P(V_5))) 
 \ar[d]_-{\pi'_1} 
\\
\Bl_{D_1(\varphi)}(B \times \P(V_5)) 
  \ar[dr]^-{\pi_1}&& \Bl_{D_2(\varphi)}(B \times \P(V_5)) 
  \ar[dl]_-{\pi_2}
 \\
& B \times \P(V_5),
}}
\end{equation}
where all the maps are blow up morphisms.
\end{lemm}

\begin{proof}
By~\cite[Tag 085Y]{stacks-project}, both varieties in the top row are canonically isomorphic to the blow up of the product of the ideals of $D_1(\varphi)$ and $D_2(\varphi)$ in $B \times \P(V_5)$.
\end{proof}

The right side of the diagram can be further simplified: since $D_2(\varphi) \subset D_1(\varphi)$,   the ideal of $\pi_2^{-1}(D_1(\varphi))$ is contained in some power of the ideal $\cI_{E_2}$
of the exceptional divisor  {$E_2$} of $\pi_2$.\
Assume   
\begin{equation}\label{eq:ideal-assumption}
\cI_{\pi_2^{-1}(D_1(\varphi))} \not\subset \bigcap_{m=0}^\infty \cI_{E_2}^m 
\end{equation} 
and let $m$ be the maximal integer such that $\cI_{\pi_2^{-1}(D_1(\varphi))} \subset \cI_{E_2}^m$
 (when $D_2(\varphi)$ is integral, $m$ is the multiplicity of $D_1(\varphi)$ along~$D_2(\varphi)$).\
Since $\cI_{E_2}$ is an invertible ideal, we can write
\begin{equation}\label{eq:strict}
\cI_{\pi_2^{-1}(D_1(\varphi))} = 
\cI' \cdot \cI_{E_2}^m
\end{equation}
for some ideal $\cI' \not\subset \cI_{E_2}$.
We let $D_1'(\varphi)$ be the subscheme of $\Bl_{D_2(\varphi)}(B \times \P(V_5))$ defined by the ideal $\cI'$.\
 {If $D_1(\varphi)$ is integral, the subscheme $D'_1(\varphi)$ contains the strict transform of $D_1(\varphi)$ in $\Bl_{D_2(\varphi)}(B \times \P(V_5))$.} 

Invertible factors of an ideal do not affect the result of the blow up, hence 
\begin{equation}\label{eq:iso-blowups}
\Bl_{\pi_2^{-1}(D_1(\varphi))}(\Bl_{D_2(\varphi)}(B \times \P(V_5))) \cong \Bl_{D_1'(\varphi)}(\Bl_{D_2(\varphi)}(B \times \P(V_5))).
\end{equation}
The next lemma describes (under some conditions) the restriction of $D_1'(\varphi)$ to the exceptional divisor of the blow up $\pi_2$.

\begin{lemm}\label{lemma:dprime-restricted}
Assume $D_3(\varphi) = \vide$.\ 
Let $\cK$ and $\cC$ be the kernel and the cokernel sheaves of the map~$\varphi$ restricted to $D_2(\varphi)$, so that both are locally free of respective ranks $2$ and $3$.\ 
The exceptional divisor $E_2$ of the blow up $\pi_2$  naturally embeds into $ \P
(\cK^\vee \otimes \cC)$ 
and we have an inclusion of schemes
\begin{equation*}
D_1'(\varphi) \cap E_2  \subset
(\P(\cK^\vee) \times_{D_2(\varphi)} \P(\cC)) \cap E_2,
\end{equation*}
 {which is} an equality if the multiplicity of $D_1(\varphi)$ along $D_2(\varphi)$ equals $2$. 
\end{lemm}

\begin{proof}
 Take a point $b \in D_2(\varphi)$.\  
Restricting to a sufficiently small neighborhood of $b$, we may assume that the bundles $\cR$ and $\Omega_{\P(V_5)}(2)$ are trivial and, 
choosing their trivializations appropriately, that the map $\varphi \colon \cR \to \Omega_{{\P(V_5)}}(2)$ is given by a matrix
\begin{equation*}
 {\left(
\begin{smallmatrix}
\varphi_{11} & \varphi_{12} & 0 \\
\varphi_{21} & \varphi_{22} & 0 \\
\varphi_{31} & \varphi_{32} & 0 \\
0 & 0 & 1
\end{smallmatrix}
\right),}
\end{equation*}
where the regular functions $\varphi_{ij} $ all vanish at $b$.\ 
In this neighborhood, the ideal of $D_2(\varphi)$ is generated by the 6 functions 
$\{\varphi_{ij}\}_{1 \le i \le 3,\, 1 \le j \le 2}$.\ 
 This shows that $E_2$ embeds into a $\P^5$-bundle over $D_2(\varphi)$ which can be   identified with $\P
 (\cK^\vee \otimes \cC)$.\ 
On the other hand, the ideal of $D_1(\varphi)$ is generated in this neighborhood by the functions
\begin{equation*}
 \varphi_{11}\varphi_{22} - \varphi_{12}\varphi_{21},\  
\varphi_{11}\varphi_{32} - \varphi_{12}\varphi_{31} ,\ 
\varphi_{21}\varphi_{32} - \varphi_{22}\varphi_{31}.
\end{equation*}
If 
 $\{u_{ij}\}_{1 \le i \le 3,\, 1 \le j \le 2}$ 
are the natural vertical homogeneous coordinates on the $\P^5$-bundle,
the preimages of these functions are  the products of
\begin{equation*}
 u_{11}u_{22} - u_{12}u_{21},\  
u_{11}u_{32} - u_{12}u_{31},\ 
u_{21}u_{32} - u_{22}u_{31}
\end{equation*}
with the square of an equation of the exceptional divisor.\ These quadratic polynomials  cut out
in $\P(\cK^\vee \otimes \cC)$ the fiber product $\P(\cK^\vee) \times_{D_2(\varphi)} \P(\cC)$ and they vanish by definition on $D_1'(\varphi)$
 {(and if the multiplicity equals 2, they generate the restriction of the ideal $\cI'$ to $E_2$).}
The lemma follows.
\end{proof}

\section{GM sixfolds}
\label{section:gm-6folds}

We now consider the family of quintic del Pezzo threefolds $q_X\colon \P(\rS_X) \to B$ constructed in Lemma~\ref{lemma:fibers-q}, 
where $B$ is a smooth 6-dimensional quadric, the map $\varphi$ defined in \eqref{eq:varphi}, and its degeneracy loci $D_k(\varphi)\subset B \times \P(V_5)$.
The intersections $\P({R}_b) \cap \Gr(2,V_5^\vee)$ are finite by Lemma~\ref{lemma:fibers-q}, hence $D_3(\varphi) = \emptyset$ by Lemma~\ref{lemma:d3}.
Let us describe $D_2(\varphi)$.

\begin{lemm}\label{lemma:d2}
Assume that~$\eqref{eq:assumptions}$ holds.\
There is a $\P^1$-fibration $D_2(\varphi) \to \Hilb^2_{Q'}(Y)$.\
In particular, $D_2(\varphi)$ is smooth of \textup{(}expected\textup{)} codimension $6$ in $ B \times \P(V_5)$.
\end{lemm}

\begin{proof}
Assume that the rank of $\varphi$ at   $(b,v)\in B \times \P(V_5)$ is  1.\ 
By~\eqref{eq:ker-varphi-bv}, we have 
\begin{equation*}
\dim({R}_b \cap \bw2(v^\perp)) = 2.
\end{equation*}
The  scheme $\xi_{b,v}:=\P({R}_b \cap \bw2(v^\perp)) \cap \Gr(2,v^\perp)$ is the intersection in $\P(  \bw2(v^\perp)) $ of a line and  a quadric, hence is either a line or a scheme of length~2.\ 
Since it is contained in~$Y$, the first case is impossible and we have a well-defined map 
\begin{equation}\label{eq:map-d2-hilb}
D_2(\varphi) \to \Hilb^2(Y),
\qquad 
(b,v) \mapsto \xi_{b,v}.
\end{equation}
The line $\langle \xi_{b,v}\rangle $ spanned by $\xi_{b,v}$ is   $\P({R}_b \cap \bw2(v^\perp))$; it is contained in $\P({R}_b)$,   hence in $Q'$ by~\eqref{diagram:qp}.\
Thus the map~{\eqref{eq:map-d2-hilb}} factors through $\Hilb^2_{Q'}(Y)$.\
Furthermore, it lifts to a map 
\begin{equation*}
{\delta}\colon D_2(\varphi) \to \Hilb^2_{Q'}(Y) \times_{\OGr(2,W')} \OFl(2,3;W')
\qquad 
(b,v) \mapsto (\xi_{b,v}, {R}_b ).
\end{equation*}
Since $\OFl(2,3;W')$ is a $\P^1$-bundle over $\OGr(2,W')$, 
the lemma will follow if we show that ${\delta}$ is an isomorphism.\ We construct an inverse.

Let   $\xi$ be a point of $\Hilb^2_{Q'}(Y)$ and  let $\P({R})  \subset Q'$ be a plane containing {the line $\langle \xi \rangle \subset Q'$}.\ 
If~$\cU'$ is the tautological   bundle on $\Gr(2,V_5^\vee)$, the evaluation map 
\begin{equation*}
V_5 = H^0(\Gr(2,V_5^\vee),{\cU}^{\prime\vee}) \xrightarrow{\ \ev_\xi\ } H^0(\xi, \cU^{\prime\vee}\vert_\xi) \cong \C^4
\end{equation*}
is surjective: if not, the line $\langle \xi \rangle $ is contained in $\Gr(2,V_5^\vee)$,  and in $Q'$ by definition of $\Hilb^2_{Q'}(Y)$,  hence it is contained in $Y$,     contradicting~\eqref{eq:assumptions}.\ 
 {If $v(\xi) \in V_5$ is a   generator of the kernel of~$\ev_\xi$, we} have $\xi \subset \Gr(2,v(\xi)^\perp) \subset \Gr(2,V_5^\vee)$.\ 
On the other hand, since   $B = \OGr(3,W')$ and $\cR \subset W' \otimes \cO_B$ is the tautological subbundle, there is a unique point $b({R}) \in B$ such that ${R} = {R}_{b({R})} $.\ 
The intersection $\P({R}_{b({R})} \cap \bw2(v(\xi)^\perp))$ contains the line $\langle \xi \rangle $, hence ${R}_{b({R})} \cap \bw2(v(\xi)^\perp)$ is at least 2-dimensional.\ 
Thus, the association   $(\xi,{R}) \mapsto (b({R}),v(\xi))  \in B\times \P(V_5) $ defines a map
$
\Hilb^2_{Q'}(Y) \times_{\OGr(2,W')} \OFl(2,3;W') \to D_2(\varphi) 
$
which is inverse to ${\delta}$.
 \end{proof}

The next step is a description of $D_1(\varphi)$ and $D_1'(\varphi)$.\ 
Recall that by definition, we have $\P(\rS_Y) \subset \P(\rS_{Q'}) = \P_B(\cR)$.

\begin{prop}\label{proposition:d}
There is 
a proper birational map 
\begin{equation*}
\rho \colon \Bl_{\P(\rS_Y)}(\P_B(\cR)) \to D_1(\varphi)
\end{equation*}
which is an isomorphism over the complement of $D_2(\varphi)$ and   a $\P^1$-fibration over $D_2(\varphi)$.\
In particular, $D_1(\varphi)$ is irreducible of \textup{(}expected\textup{)} codimension $2$ in $ B \times \P(V_5)$, smooth outside of~$D_2(\varphi)$.
\end{prop}

\begin{proof}
Since $\P(\rS_Y)=p_{Q'}^{-1}(Y)$, we have
\begin{equation*}
\Bl_{\P(\rS_Y)}(\P_B(\cR)) =  \P_B(\cR) \times_{Q'} \Bl_Y(Q').
\end{equation*}
Moreover, since $Y$ is the transversal intersection of $Q'$ and $\Gr(2,V_5^\vee)$ in $\P(\bw2V_5^\vee)$, we have $\Bl_Y(Q') = Q' \times_{\P(\sbw2V_5^\vee)} \Bl_{\Gr(2,V_5^\vee)}(\P(\bw2V_5^\vee))$.
Finally,  one checks that
\begin{equation*}
\Bl_{\Gr(2,V_5^\vee)}(\P(\bw2V_5^\vee)) \cong \P(\Omega_{ \P(V_5)}^2(2)),
\end{equation*}
where the blow up map is induced by the embedding  $\Omega_{\P(V_5)}^2(2) \hookrightarrow \bw2V_5^\vee \otimes \cO_{\P(V_5)}$.\
All this implies
\begin{equation*}
\Bl_{\P(\rS_Y)}(\P_B(\cR)) \cong \P_B(\cR) \times_{\P(\sbw2V_5^\vee)} \P(\Omega_{\P(V_5)}^2(2)),
\end{equation*}
hence $\Bl_{\P(\rS_Y)}(\P_B(\cR))$ parameterizes triples $(b,v,\eta)  \in B \times \P(V_5) \times \P(\bw2V_5^\vee) $ such that \mbox{$\eta \in \P({R}_b \cap \bw2(v^\perp)) = \P(\Ker\varphi_{b,v})$}.
This means that the map 
\begin{equation*}
\Bl_{\P(\rS_Y)}(\P_B(\cR)) \to B \times \P(V_5),
\qquad
(b,v,\eta) \mapsto (b,v)
\end{equation*}
factors through a surjective map onto $D_1(\varphi)$.\ It  is an isomorphism over $D_1(\varphi) \setminus D_2(\varphi)$ and
the scheme-theoretic preimage of $D_2(\varphi)$ is isomorphic to the projectivization of the rank-2 vector bundle $\cK := \Ker(\varphi\vert_{D_2(\varphi)})$.\ 
This proves the proposition.
\end{proof}

{We now analyze the right  side of the diagram~\eqref{eq:blowups}.\
In order to use the isomorphism~\eqref{eq:iso-blowups}  for the blow up $\pi'_1$, we need a description of the subscheme $D'_1(\varphi) \subset \Bl_{D_2(\varphi)}(B \times \P(V_5))$  defined after~\eqref{eq:strict} (the assumption~\eqref{eq:ideal-assumption} holds in our case since the scheme $\Bl_{D_2(\varphi)}(B \times \P(V_5))$ is  integral  and $D_1(\varphi)$ is non-empty).} This description is provided by the next proposition.

\begin{prop}\label{prop:dprime}
The subscheme $D'_1(\varphi) \subset \Bl_{D_2(\varphi)}(B \times \P(V_5))$ is isomorphic to the blow up of $\Bl_{\P(\rS_Y)}(\P_B(\cR))$ along $\rho^{-1}(D_2(\varphi))$.\ 
It is in particular smooth and irreducible.
\end{prop}

\begin{proof}
{Consider the composition
\begin{equation*}
\Bl_{\rho^{-1}(D_2(\varphi))}(\Bl_{\P(\rS_Y)}(\P_B(\cR))) \to \Bl_{\P(\rS_Y)}(\P_B(\cR)) \xrightarrow{\ \rho\ } D_1(\varphi) \hra B \times \P(V_5).
\end{equation*}
The preimage of $D_2(\varphi) \subset B \times \P(V_5)$ is a Cartier divisor, hence the composition lifts to a map 
\begin{equation}\label{eq:map-dprime}
\Bl_{\rho^{-1}(D_2(\varphi))}(\Bl_{\P(\rS_Y)}(\P_B(\cR))) \to \Bl_{D_2(\varphi)}(B \times \P(V_5)).
\end{equation}
By Proposition~\ref{proposition:d}, $\rho$ is an isomorphism over $D_1(\varphi) \setminus D_2(\varphi)$, hence 
 {the image of~\eqref{eq:map-dprime} is contained in the strict transform of $D_1(\varphi)$, hence a fortiori in $D'_1(\varphi)$.
So,}
the map~\eqref{eq:map-dprime} factors through a map
\begin{equation*}
\Bl_{\rho^{-1}(D_2(\varphi))}(\Bl_{\P(\rS_Y)}(\P_B(\cR))) \to D'_1(\varphi)
\end{equation*}
which is an isomorphism over the complement of the exceptional divisor $E_2$ of  {the blow up} $\Bl_{D_2(\varphi)}(B \times \P(V_5))$.\ 
We now prove that this map is an isomorphism.

By Lemma \ref{lemma:d2} and Proposition \ref{proposition:d}, its source is a smooth variety.\ 
Let us check that its target is smooth as well.\
By Proposition \ref{proposition:d}, $D'_1(\varphi)$ is smooth {of codimension~2} over the complement of~$D_2(\varphi)$.\
On the other hand, its intersection with 
$E_2$ 
  {is by Lemma~\ref{lemma:dprime-restricted} contained in} a $(\P^1 \times \P^2)$-fibration over $D_2(\varphi)$
(note that  {in our case,} $E_2$ equals   the $\P^5$-bundle over~$D_2(\varphi)$ since $D_2(\varphi) \subset B \times \P(V_5)$ is smooth of codimension~6 by Lemma~\ref{lemma:d2}).\
In particular, {the codimension of $D'_1(\varphi) \cap E_2$ in $E_2$ is greater than or equal to 2.\
But $E_2$ is a Cartier divisor, hence the codimension of $D'_1(\varphi) \cap E_2$ in $E_2$ does not exceed the codimension of $D'_1(\varphi)$ in $\Bl_{D_2(\varphi)}(B \times \P(V_5))$.\
Altogether, this shows that the codimension of $D'_1(\varphi) \cap E_2$ in $E_2$ is 2 and that
$D'_1(\varphi) \cap E_2$ equals the $(\P^1 \times \P^2)$-fibration over $D_2(\varphi)$.\
In particular,  it is smooth.\
Since~$E_2$ is a Cartier divisor, it follows that $D'_1(\varphi)$ is smooth along $E_2$.\
Since it is smooth  outside of~$E_2$, it is  smooth everywhere.}
 
Finally,   the preimage of $D_2(\varphi)$ in $\Bl_{\rho^{-1}(D_2(\varphi))}(\Bl_{\P(\rS_Y)}(\P_B(\cR)))$
is by Proposition~\ref{proposition:d}   an irreducible divisor.\
 It remains to note that a proper morphism between two smooth projective varieties  which is an isomorphism
between the complements of two irreducible divisors induces an isomorphism on the Picard groups,   hence is an isomorphism.}
\end{proof}

We can now prove our main result.
 {We denote by $\cU_X$ the pullback to $X$ of the tautological rank-2 bundle on $\Gr(2,V_5)$.}

\begin{theo}\label{theorem:main}
Assume that~$\eqref{eq:assumptions}$ holds and let $X$ be a GM sixfold which is generalized dual to a GM surface~$Y$, as described in Section~{\textup{\ref{subsection:dual-gm}}}.\
There is a diagram
\begin{equation*}
\xymatrix@C=-1pt@R=3ex{
&& \Bl_{D'_1(\varphi)}(\Bl_{D_2(\varphi)}(B \times \P(V_5))) \ar[dl] \ar[dr] \\
& \P(\rS_X) \times_X \P_X(\cU_X) \ar[dl] && \Bl_{D_2(\varphi)}(B \times \P(V_5)) \ar[dr] \\
X &&&& B \times \P(V_5),
}
\end{equation*}
where the leftmost map is a fibration with fiber $\P^3 \times \P^1$  and the  top left arrow is 
 the blow up with center  a $\P^2$-fibration over $D_2(\varphi)$.

Furthermore, the integral cohomology of $X$ embeds into the sum of Tate twists of $\Z$, $H^\bullet(Y,\Z)$, and $H^\bullet(\Hilb^2_{Q'}(Y),\Z)$.\
In particular, it is torsion-free.
\end{theo}

\begin{proof}
{Consider the family of quintic del Pezzo threefolds $\cM = \P(\rS_X) \to B$ constructed in Lemma~\ref{lemma:fibers-q}.\ Since the bundle $\cU$ on $\cM$ is the pullback of the bundle $\cU_X$ from $X$, we have
\begin{equation*}
\P_\cM(\cU) \cong \P(\rS_X) \times_X \P_X(\cU_X).
\end{equation*}
Let $\varphi  $
be the   map of bundles on $ B \times \P(V_5)$ defined in~\eqref{eq:varphi}.\
By Lemma~\ref{lemma:fibers-q} and Lemma~\ref{lemma:d3}, we have $D_3(\varphi) = 0$.\
Hence, by Lemma \ref{lemma:d2} and Proposition \ref{proposition:d}, the assumptions of Proposition~\ref{proposition:quintic-fibration} are fulfilled and we obtain
\begin{equation*}
\P(\rS_X) \times_X \P_X(\cU_X) \cong \Bl_{D_1(\varphi)}(B \times \P(V_5)).
\end{equation*}}
Therefore,  a combination of Lemma~\ref{lemma:blow ups} and isomorphism~\eqref{eq:iso-blowups} shows the existence of
the   top left arrow  and proves that it is the blow up of the preimage of $D_2(\varphi)$.\
Since  $D_3(\varphi) = \emptyset$, the latter is, by Proposition~\ref{proposition:quintic-fibration}, a $\P^2$-fibration over $D_2(\varphi)$.

By the projective bundle formula, $H^\bullet(X,\Z)$   embeds  into $H^\bullet(\P(\rS_X) \times_X \P_X( {\cU_X}),\Z)$
 and by the blow up formula and the fact that $D_2(\varphi)$ is smooth, the latter embeds into the cohomology of the top variety.  

It remains to describe  {$H^\bullet(\Bl_{D'_1(\varphi)}(\Bl_{D_2(\varphi)}(B \times \P(V_5))), \Z)$}.
We use its description as an iterated blow up.
Since all blow up centers are smooth  {by Lemma~\ref{lemma:d2} and Proposition~\ref{prop:dprime}}, the cohomology of the top variety 
is a direct sum of Tate twists of $H^\bullet(B\times \P(V_5))$ and of  the cohomology of the blow up centers.\ 
Explicitly, we have
 \begin{equation*}
\begin{array}{rccl}
H^\bullet(\Bl_{D'_1(\varphi)}(\Bl_{D_2(\varphi)}(B \times \P(V_5))),\Z) & = &&
H^\bullet(B\times \P(V_5),\Z) \\
&&  \oplus& 
H^\bullet(D_2(\varphi),\Z) \otimes (\LL \oplus \LL^2 \oplus \LL^3 \oplus \LL^4 \oplus \LL^5) \\
&&  \oplus&
H^\bullet(D'_1(\varphi),\Z) \otimes \LL,
\end{array}
\end{equation*}
where $\LL$ is the Tate module.\ Furthermore,  we have by Lemma~\ref{lemma:d2}
\begin{equation*}
H^\bullet(D_2(\varphi),\Z) = H^\bullet(\Hilb^2_{Q'}(Y),\Z) \otimes (\unit \oplus \LL).
\end{equation*}
Since $D'_1(\varphi)$ is itself an iterated blow up,  we have by Proposition~\ref{prop:dprime}
\begin{equation*}
\begin{array}{rccl}
H^\bullet(D'_1(\varphi),\Z) & =&& 
H^\bullet(\P_B( {\cR}),\Z) \\
&&  \oplus&
H^\bullet(\P(\rS_Y),\Z) \otimes (\LL \oplus \LL^2) \\
&&  \oplus&
H^\bullet(D_2(\varphi),\Z) \otimes (\unit \oplus \LL) \otimes (\LL \oplus \LL^2).
\end{array}
\end{equation*}
Combining all these formulas, we deduce  {that $H^\bullet(\Bl_{D'_1(\varphi)}(\Bl_{D_2(\varphi)}(B \times \P(V_5))),\Z)$ equals}
\begin{equation*}
\begin{array}{rccl}
&&& H^\bullet(B,\Z) \otimes (\unit \oplus 2\LL \oplus 2\LL^2 \oplus 2\LL^3 \oplus \LL^4) \\
&& \oplus &
H^\bullet(Y,\Z) \otimes (\LL^2 \oplus 2\LL^3 \oplus 2\LL^4 \oplus 2\LL^5 \oplus \LL^6) \\
&& \oplus & 
H^\bullet(\Hilb^2_{Q'}(Y),\Z) \otimes (\LL \oplus 3\LL^2 \oplus 5\LL^3 \oplus 5\LL^4 \oplus 3\LL^5 \oplus \LL^6),
\end{array}
\end{equation*}
Since $B \cong Q^6$, the first is a sum of Tate twists of $\Z$.

To prove the last statement, we note that   $H^\bullet(Y,\Z)$ is torsion-free because $Y$ is a K3 surface.\ 
 By Lemma~\ref{lemma:s2q},  $\Hilb^2_{Q'}(Y)$ is a smooth ample divisor in $\Hilb^2(Y)$.\
Since $H^\bullet(\Hilb^2_{Q'}(Y),\Z)$ is torsion-free,  the same is true for $H^\bullet(\Hilb^2_{Q'}(Y),\Z)$ by the Lefschetz Hyperplane Theorem and the Universal Coefficients Theorem.\
This completes the proof of the theorem.
 \end{proof}

The argument in the proof above  also works at the level of Chow motives and gives an  {isomorphism}
\begin{align*}
(\Bl_{D'_1(\varphi)}(\Bl_{D_2(\varphi)}(B \times \P(V_5)))) & \cong
(\MM(B) \oplus (\MM(Y) \otimes \LL^2)) \otimes (\unit \oplus 2\LL \oplus 2\LL^2 \oplus 2\LL^3 \oplus \LL^4) \\
& {}\oplus 
\MM(\Hilb^2_{Q'}(Y)) \otimes (\LL \oplus 3\LL^2 \oplus 5\LL^3 \oplus 5\LL^4 \oplus 3\LL^5 \oplus \LL^6),
\end{align*}
where $\MM(-)$ stands for the integral Chow motive of a variety.\
On the other hand,  {since} the map $\Bl_{D'_1(\varphi)}(\Bl_{D_2(\varphi)}(B \times \P(V_5))) \to \P(\rS_X) \times_X \P_X(\cU)$ 
is the blow up with center a $\P^2$-bundle over $D_2(\varphi)$,  {we have}
\begin{align*}
\MM(\Bl_{D'_1(\varphi)}(\Bl_{D_2(\varphi)}(B \times \P(V_5)))) & \cong
\MM(X) \otimes (\unit \oplus 2\LL \oplus 2\LL^2 \oplus 2\LL^3 \oplus \LL^4) \\
& \oplus 
\MM(\Hilb^2_{Q'}(Y)) \otimes (\LL \oplus 3\LL^2 \oplus 5\LL^3 \oplus 5\LL^4 \oplus 3\LL^5 \oplus \LL^6).
\end{align*}
Comparing the two expressions,   we obtain
\begin{equation}\label{eq:motives}
\MM(X) \otimes \MM_1 \oplus \MM_2 \cong
(\MM(B) \oplus (\MM(Y) \otimes \LL^2)) \otimes \MM_1 \oplus \MM_2,
\end{equation} 
where $\MM_1 = \unit \oplus 2\LL \oplus 2\LL^2 \oplus 2\LL^3 \oplus \LL^4$ and $\MM_2 = \MM(\Hilb^2_{Q'}(Y)) \otimes (\LL \oplus 3\LL^2 \oplus 5\LL^3 \oplus 5\LL^4 \oplus 3\LL^5 \oplus \LL^6)$.
It would be interesting to understand whether the summand $\MM_2$ and the factor $\MM_1$ can be canceled out, providing an isomorphism of motives \
$\MM(X) \cong \MM(B) \oplus (\MM(Y) \otimes \LL^2)$?

In any case,  any realization functor with values in a semisimple category  {such that the realization of $\MM_2$ is not a zero divisor,} when evaluated on $\MM(X)$,  satisfies the above equality.\
For instance, we have the following  result.

\begin{coro}
The Hodge numbers of any smooth GM sixfold $X$ satisfy
\begin{equation}\label{hpq}
h^{p,q}(X) = h^{p,q}(B) + h^{p-2,q-2}(Y).
\end{equation}
 {In particular, the Hodge diamond of a smooth GM sixfold  {is}
\begin{equation*}
\begin{smallmatrix}
&&&&&& 1 \\
&&&&& 0 && 0 \\
&&&& 0 && 1 && 0 \\
&&& 0 && 0 && 0 && 0 \\
&& 0 && 0 && 2 && 0 && 0  \\
& 0 && 0 && 0 && 0 && 0 && 0 \\
0 && 0 && 1 && 22 && 1 && 0 && 0 \\
& 0 && 0 && 0 && 0 && 0 && 0 \\
&& 0 && 0 && 2 && 0 && 0  \\
&&& 0 && 0 && 0 && 0 \\
&&&& 0 && 1 && 0 \\
&&&&& 0 && 0 \\
&&&&&& 1 \\
\end{smallmatrix}
\end{equation*}}
\end{coro}

\begin{proof}
For the GM sixfolds constructed in Lemma~\ref{lemma:6fold}, the equality~\eqref{hpq} follows from~\eqref{eq:motives} by considering the Hodge realization functor; for arbitrary GM sixfolds, it follows by a deformation argument.\
It remains to note that $B$ is a 6-dimensional quadric to obtain all the Hodge numbers.
\end{proof}

\end{document}